\documentclass[letterpaper, 10 pt, conference]{ieeeconf}  

\IEEEoverridecommandlockouts                              

\usepackage{epsfig} 
\usepackage{mathptmx} 
\usepackage{times} 
\usepackage{amsmath} 
\usepackage{amssymb}  

\usepackage{amsfonts}
\usepackage{mathtools}
\usepackage{graphicx}
\usepackage{enumerate}
\usepackage[noadjust]{cite}
\usepackage{hyperref}

\title{\LARGE \bf
Optimal Distributed \texorpdfstring{$H_{\infty}$}{H-infinity} State Feedback for Systems \\ with Symmetric and Hurwitz State Matrix
}

\author{Carolina Lidstr\"om and Anders Rantzer
\thanks{This research was supported by the Swedish Research Council through the LCCC Linnaeus Center and by the Swedish Foundation for Strategic Research through the project ICT-Psi. Both authors are with the Department of Automatic Control, Lund University, Box 118, {SE-221 00 Lund}, Sweden. {\tt\small carolina.lidstrom@control.lth.se}  {\tt\small anders.rantzer@control.lth.se} }
}

\begin{document}

\newtheorem{theo}{Theorem}
\setcounter{theo}{0}
\newtheorem{lem}{Lemma}
\setcounter{lem}{0}
\newtheorem{example}{Example}
\setcounter{example}{0}
\newtheorem{remark}{Remark}
\setcounter{example}{0}

\maketitle
\thispagestyle{empty}
\pagestyle{empty}

\begin{abstract}
We address \texorpdfstring{$H_{\infty}$}{H-infinity} structured static state feedback and give a simple form for an optimal control law applicable to linear time invariant systems with symmetric and Hurwitz state matrix. More specifically, the control law as well as the minimal value of the norm can be expressed in the matrices of the system's state space representation, given separate cost on state and control input. Thus, the control law is transparent, easy to synthesize and scalable. Furthermore, if the plant possess a compatible sparsity pattern it is also distributed. Examples of such sparsity patterns are included. Furthermore, we give an extension of the optimal control law that incorporate coordination among subsystems. We demonstrate by a numerical example that the derived optimal controller is equal in performance to an optimal controller derived by the riccati equation approach.
\end{abstract}

\section{Introduction}
Systems with a high density of sensors and actuators often lack centralized information and computing capability. Thus, structural constraints, e.g,  on information exchange among subsystems, have to be incorporated into the design procedure of the control system. However, imposing such constraints may greatly complicate controller synthesis. 

We address \texorpdfstring{$H_{\infty}$}{H-infinity} structured static state feedback, a problem that is recognized as genuinely hard given arbitrary plant and controller structures. However, we give a simple form for an optimal control law applicable to linear time invariant (LTI) systems with symmetric and Hurwitz state matrix. Furthermore, if the system possess a compatible sparsity pattern the proposed controller is distributed. Consider the following LTI system
\begin{equation}
\dot{x} = \underbrace{-\textrm{diag}(1,3,2)}_Ax+\underbrace{\begin{bmatrix}
-1 & 0 &0\\
1 &1 &-1\\
0 &0 &1
\end{bmatrix}}_Bu + w \label{bufferSys}
\end{equation}
where the state $x$, the control input $u$ and the disturbance $w$ are real valued. The static state feedback controllers 
\begin{displaymath}
L_1 =\begin{bmatrix}
1 &-\frac{1}{3} &0\\
0 &-\frac{1}{3} &0\\
0 &\frac{1}{3} &-\frac{1}{2}
\end{bmatrix} \textrm{ and }
L_2  =\begin{bmatrix}
   0.93 & -0.11&
  0.00 \\   -0.05 &
   -0.17 &   -0.01 \\
  0.04 &
  0.16 &
  -0.26
\end{bmatrix}
\end{displaymath}
both minimize the \texorpdfstring{$H_{\infty}$}{H-infinity} norm of the closed-loop system from disturbance $w$ to performance output $\left(x,u\right)$, i.e., when ${u = L_1x}$ and $u = L_2x$, respectively. However, they have different structural properties, e.g., $L_1$ is sparser than $L_2$. Furthermore, the feedback law $u=L_1x$ is distributed as the matrix $L_1$ has the same structure as the sparse matrix $B^T$. This is not the case for controller $L_2$. Controller $L_1$ can be given on the simple form we propose. More specifically $L_1$ can be written as $L_1=B^TA^{-1} $. Controller $L_2$ is derived by the algebraic riccati equation (ARE) approach. That is, iteration over an ARE-constraint until the minimal value of the norm is obtained, see \cite{zhou1996robust} for details. Controllers synthesized by the ARE method are often dense, as is the case for controller $L_2$. Moreover, as the control law we give, i.e., $u= B^TA^{-1}x$, is optimal, it is equal in performance to any centrally derived optimal controller. Additionally, it is transparent in its structure, easy to synthesize and scalable.  

In the 1980's, synthesis of controllers that achieve $H_{\infty}$ norm specifications became a major research area and was formulated in \cite{zames1981feedback}. The solution to the synthesis problem was initially based on operator design but evolved to a state-space based design that paved the way for optimization tools to be used, e.g., see~\cite{doyle1989state}. The \texorpdfstring{$H_{\infty}$}{H-infinity} norm condition can be turned into a linear matrix inequality (LMI) by the Kalman-Yakubovich-Popov lemma~\cite{gahinet1994linear}, see Lemma~\ref{KYP} in Appendix for the version used in this paper. This reformulation made the synthesis computationally easier. As the theory on \texorpdfstring{$H_{\infty}$}{H-infinity} feedback control emerged, a decentralized version took form, e.g., see \cite{zhai2001decentralized}. Imposing general sparsity constraints on the controller might complicate the design procedure. However, design is simplified if the constrained set of controllers $K$ is quadratically invariant with respect to the given system \cite{rotkowitz2006characterization}. It is also simplified if the closed-loop system is constrained to be internally positive~\cite{tanaka2011bounded}. However, in contrast to our method, the methods previously mentioned do not result in controllers that are equal in performance to the central non-structured controller of the system. 

The optimal control law $u=B^TA^{-1}x$ only requires some relatively inexpensive matrix calculations for its synthesis, especially for sparse systems. This is in relation to general \texorpdfstring{$H_{\infty}$}{H-infinity} controller synthesis where more expensive computational methods are required. Additionally, its structure is transparent, which is not often the case in \texorpdfstring{$H_{\infty}$}{H-infinity} controller synthesis.

The \texorpdfstring{$H_{\infty}$}{H-infinity} framework treats worst-case disturbance as opposed to stochastic disturbance in the $H_2$ framework. However, the transparent structure and simple synthesis of the derived optimal feedback law might motivate its use even when some characteristics of the disturbance are known. Moreover, it can be extended to incorporate coordination in a system of heterogeneous subsystems, given a linear coordination constraint. The coordinated control law is a superposition of a decentralized and a centralized part, where the latter is equal for all agents. This structure might be well suited for distributed control purposes as well. See~\cite{madjidian2014distributed} for a similar problem treated in the $H_2$ framework.

The outline of this paper is as follows. This section is ended with some notation. In Section II, the main results is stated and proved. Section III treats system sparsity patterns that result in a distributed control law. Section IV gives an extension of the control law that incorporates coordination. In Section V, the performance of our optimal control law is compared, by a numerical example, to an optimal controller synthesized by the ARE approach. Concluding remarks are given in Section VI.   

The set of real numbers is denoted $\mathbb{R}$ and the space $n$-by-$m$ real-valued matrices is denoted $\mathbb{R}^{n \times m}$. The identity matrix is written as $I$ when its size is clear from context, otherwise $I_n$ to denote it is of size $n$-by-$n$.  Similarly, a column vector of all ones is written $\mathbf{1}$ if its length is clear form context, otherwise $\mathbf{1}_n$ to denote it is of length $n$. 

For a matrix $M$, the inequality $M \geq 0$ means that $M$ is entry-wise non-negative and $M \in \mathbb{R}^{n \times n}$ is said to be Hurwitz if all eigenvalues have negative real part. The matrix $M$ is said to be Metzler if its off-diagonal entries are non-negative and the spectral norm of $M$ is denoted $\|M\|$. Furthermore, for a square symmetric matrix $M$, $M \prec 0$ ($M \preceq 0$) means that $M$ is negative (semi)definite while $M \succ 0$ ($M \succeq 0$) means $M$ is positive (semi)definite.  

The $H_{\infty}$ norm of a transfer function $F(s)$ is written as $\|F(s)\|_{\infty}$. It is well known that this operator norm equals the induced 2-norm, that is
\begin{displaymath}
\|F\|_{\infty} = \textrm{sup}_{v \neq 0} \frac{\|Fv\|_2}{\|v\|_2}.
\end{displaymath}  

\section{An optimal \texorpdfstring{$H_{\infty}$}{H-infinity} state feedback law}
Consider a LTI system of the following structure
\begin{equation}
\dot{x} = Ax+Bu+w \quad \label{G}
\end{equation} 
where state matrix $A \in \mathbb{R}^{n \times n}$ is symmetric and Hurwitz and state $x \in \mathbb{R}^{n}$ can be measured. Moreover, control input ${u \in \mathbb{R}^m}$, disturbance $w \in \mathbb{R}^n$ and matrix ${B \in \mathbb{R}^{n \times m}}$. 

Given (\ref{G}) and performance output $\left(x,\,u\right)$, consider a stabilizing static state feedback law $u \coloneqq Lx$, where $L \in \mathbb{R}^{m \times n}$. The transfer function of the closed-loop system, i.e., from disturbance $w$ to performance output $\left(x,\,u\right)$, is given by
\begin{equation}
G_L(s) = \begin{bmatrix}
I \\L
\end{bmatrix} (sI-(A+BL))^{-1}. \label{Gtransfer}
\end{equation} 
For (\ref{G}) with $A$ symmetric and Hurwitz, an optimal \texorpdfstring{$H_{\infty}$}{H-infinity} static state feedback controller $L$, i.e., a matrix $L$ such that $\| G_L  \| _{\infty}$ is minimized, can be given explicitly in the matrices $A$ and $B$. This is the main result of this paper and it is stated in the following theorem, followed by a proof.

\begin{theo} \label{theo1}
Consider the system (\ref{G}) with $A$ symmetric and Hurwitz. Then, the norm $\| G_L\| _{\infty}$ is minimized by the static state feedback controller $L_* = B^TA^{-1}$. The minimal value of the norm is $\sqrt{\|(A^2 + B^TB)^{-1} \|}$.
\end{theo}
\begin{proof}
Given $\gamma>0$, the following statements are equivalent. 
\begin{enumerate}[(i)]
\item There exists a stabilizing controller $L$ such that 
\begin{displaymath}
\| G_L \|_{\infty} = \left \| \begin{bmatrix}
I \\ L
\end{bmatrix} (i \omega -A-BL)^{-1} \right \|_{\infty} < \gamma. 
\end{displaymath} \label{th1}
\item \label{th2}There exists a matrix $P \succ 0$ such that
\begin{displaymath} 
\begin{bmatrix}
(A+BL)^TP+P(A+BL) & P & [I \enskip L^T] \\ P & -\gamma^2I &0 \\
[I \enskip L^T]^T & 0 & -I \\ 
\end{bmatrix} \prec 0. \\ 
\end{displaymath}
\item \label{th3}There exist matrices $X\succ 0$ and $Y$ such that 
\begin{displaymath}
\begin{bmatrix}
AX+XA+BY+Y^TB^T & I & [X \enskip Y^T] \\ I & -\gamma^2I &0 \\
[X \enskip Y^T]^T & 0 & -I \\ 
\end{bmatrix} \prec 0. 
\end{displaymath} 
\item \label{th4}There exist matrices $X\succ 0$ and $Y$ such that 
\begin{align*}
\left(X+A\right)^2+ \left(Y+B^T\right)^T&\left(Y+B^T\right) \\ &-A^2 -BB^T + \gamma^{-2}I \prec 0. 
\end{align*}
\item \label{th5} \begin{displaymath}
-A^2 -BB^T + \gamma^{-2}I \prec 0. 
\end{displaymath}
\item \label{th6}\begin{displaymath}
\gamma > \sqrt{\|\left(A^2 + BB^T \right)^{-1}\|}. 
\end{displaymath}
\end{enumerate}
The equivalence between (\ref{th1}) and (\ref{th2}) is given by the K-Y-P-lemma, Lemma~\ref{KYP} given in Appendix. Statement (\ref{th2}) can be equivalently written as (\ref{th3}) after right- and left-multiplication with $\textrm{diag}(P^{-1},I,I)$ and change of variables ${(P^{-1},LP^{-1}) \rightarrow (X,Y)}$. The equivalence between (\ref{th3}) and (\ref{th4}) is obtained by applying Schur's complement lemma and completion of squares to the inequality in (\ref{th3}). Choosing $X=-A$ and $Y=-B^T$ shows equivalence between (\ref{th4}) and (\ref{th5}). It is possible to choose $X=-A$ as $A$ is symmetric and Hurwitz, i.e., $A \prec 0$. Finally,  notice that $A^2+BB^T\succ 0$ and thus $\left(A^2+BB^T\right)^{-1} \succ 0$. Thus, 
\begin{displaymath}
\textrm{(\ref{th5})} \iff \gamma^2I \succ  \left(A^2+BB^T\right)^{-1}
\iff \textrm{(\ref{th6})}.
\end{displaymath}
Given $X=-A$ and $Y=-B^T$, $\gamma$ is minimized. Thus, ${L_* = YX^{-1} = B^TA^{-1}}$ minimizes the norm in (\ref{th1}) and the minimal value of the norm, i.e,  $\|G_{L_*}\|$, is less than $\gamma$. Now, define $\gamma_* \coloneqq \sqrt{\|\left(A^2+BB^T\right)^{-1}\|}$ and assume that $\|G_{L_*}\|_{\infty} \neq \gamma_*$. Then $\|G_{L_*}\|_{\infty}$ has to be strictly larger than or strictly smaller than $\gamma_*$. Consider $\|G_{L_*}\|_{\infty}>\gamma_*$. This statement contradicts statement (\ref{th1}) and (\ref{th6}) and is therefore false. Now, consider instead $\|G_{L_*}\|_{\infty} <\gamma_*$. This statement contradicts that $\gamma$ is minimized and is therefore also false. Hence, the statement $\|G_{L_*}\|_{\infty} \neq \gamma_*$ is false and 
\begin{displaymath}
\|G_{L_*}\|_{\infty} = \sqrt{\|\left(A^2+BB^T\right)^{-1}\|}.
\end{displaymath}
\end{proof}

\begin{remark} 
Instead of performance output $ \left(x, \,u\right)$ consider $ \left(\tilde{x}, \,\tilde{u}\right)$ defined as
\begin{equation}
\begin{bmatrix} \tilde{x} \\ \tilde{u} \end{bmatrix} = \begin{bmatrix}
C & 0 \\ 0 & D
\end{bmatrix}\begin{bmatrix}
x \\ u
\end{bmatrix} \label{scaled}
\end{equation}
with square matrices $C \in \mathbb{R}^{n \times n}$ and $D \in \mathbb{R}^{m\times m}$. Define matrices $Q\coloneqq C^TC\succ 0$ and $R \coloneqq D^TD \succ 0$. If $-AQ^{-1}$ is symmetric and positive definite it is still possible to find an explicit expression for an optimal static state feedback controller, specifically $L_* = R^{-1}B^TQA^{-1}$. Hence, other performance objectives than $\left(x,\,u\right)$ can be included in this framework. However, cross terms between $x$ and $u$ are not possible for Theorem~\ref{theo1} to hold. Matrices $Q$ and $R$ can be seen as the cost matrices for state $x$ and control input $u$, respectively, and used as design parameters in the synthesis of the optimal $H_{\infty}$ static state feedback controller $L_*$. The restrictions on $Q$ and $R$, i.e., $Q \succ 0$ and $R \succ 0$, are not significantly more conservative than restrictions given on similar design matrices in static state feedback synthesis by linear quadratic control, see~\cite{athans1971role} for comparison.
\end{remark}

Synthesis of an optimal static state feedback controller $L$ that minimizes the \texorpdfstring{$H_{\infty}$}{H-infinity} norm of (\ref{Gtransfer}) generally requires additional computation beyond what is needed to compute $L_*$ given by Theorem~\ref{theo1}, i.e., some relatively simple matrix calculations. Moreover, optimal controllers generated by other methods than Theorem~\ref{theo1} are rarely as transparent as $L_*$. The transparency simplifies analysis of the structure of the optimal controller as well as enables scalability. This will be exploited in the following section. 

In order for Theorem~\ref{theo1} to be applicable, the system of interest has to have a state space representation with symmetric and Hurwitz state matrix $A$. The symmetry property of $A$ demands that states that affect each other does so with equal rate coefficient. Such representations appear, for instance, in buffer networks and models of temperature dynamics in buildings. We will now give an example of the latter. 

\begin{example}
Consider a building with three rooms as depicted in Fig.~\ref{fig:RoomTemp}. The average temperature $T_i$ in each room $i=$ 1, 2 and 3, around some steady state, is given by the following model
\begin{align}
&\dot{T}_1 = -r_1T_1+r_{12}\left(T_2-T_1\right)+u_1+w_1\nonumber \\ 
&\dot{T}_2 = -r_2T_2+r_{12}\left(T_1-T_2\right)+r_{23}\left(T_3-T_2\right)+u_2+w_2 \label{temp} \\
&\dot{T}_3 = -r_3T_3+r_{23}\left(T_2-T_3\right)+u_3+w_3 \nonumber
\end{align} 
governed by heat balance. The parameters $r_\bullet $ are constant, real-valued and positive. They are the rate coefficients of the system. For instance, $r_{12}$ is the rate coefficient of the heat transfer through the wall between room 1 and 2. Changes in outdoor temperature and disturbances specific for each room, such as a window is opened, are modeled by disturbances $w_i$. The average temperatures can be measured as well as controlled through heating and cooling devices, given by control inputs $u_i$. If (\ref{temp}) is written on form (\ref{G}), it is easy to see that matrix $A$ is symmetric. Thus, Theorem~\ref{theo1} is applicable to (\ref{temp}), assuming that that parameters $r_\bullet$ are such that $A$ is also Hurwitz. Given a disturbance, the feedback law with $L_*$ from Theorem~\ref{theo1} tries to keep the average temperature as close to the steady state as possible while minimizing the cost that comes with heating and cooling.
\end{example} 

\begin{figure}[t] 
\centering
\vspace{0.3cm}
\includegraphics[width=0.5\linewidth]{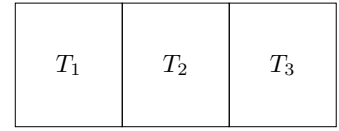}
\caption{Schematic of a building with three rooms. The average temperature in each room $i=1,2$ and $3$ is denoted $T_i$ and given by (\ref{temp}).  \label{fig:RoomTemp}}
\end{figure}

\section{Distributiveness and scalability}
The structure of the optimal controller $L_*$ given in Theorem~\ref{theo1} is clearly dependent on the structure of matrices $A$ and $B$ in (\ref{G}). For instance, if $A$ is diagonal and $B$ is sparse, $L_*$ has the same sparsity pattern as $B^T$. Moreover, controller $L_*$ is distributed if (\ref{G}) possess a compatible sparsity pattern. It is worthwhile to point out that for some sparsity patterns of (\ref{G}) the representation $L_*^{-1}u=x$ instead of $u = L_*x$ might be beneficial for computation of $u$, that is if $B^T$ is invertible. We will begin this section with an example to demonstrate the distributiveness and scalability of the optimal controller given a system of structure (\ref{G}) with compatible sparsity pattern. 
\begin{example}
Consider the following LTI system, containing three subsystems denoted $S_1$, $S_2$ and $S_3$,  
\begin{align}
&S_1: \quad \dot{x}_1 = A_1x_1 + B_1u_1+w_1 \nonumber \\
&S_2: \quad \dot{x}_2 = A_2x_2 + B_2u_1+B_3u_2 +w_2 \label{Garea} \\
&S_3: \quad \dot{x}_3 = A_3x_3 + B_4u_2 +w_3 \nonumber 
\end{align} 
where each subsystem $S_i$, $i=$1, 2 and 3, has finite state dimension $n_i \geq 1$, each control input $u_i$, $i=$1, 2 and 3, is a vector of finite length $m_i \geq 1$ and the matrices are of suitable dimension. Furthermore, matrices $A_1$, $A_2$ and $A_3$ are assumed to be symmetric and Hurwitz. Then,  Theorem~\ref{theo1} is applicable to (\ref{Garea}) and results in the optimal controller
\begin{equation}
L_* = \begin{bmatrix}
B_1^TA_1^{-1} & B_2^TA_2^{-1} & 0 \\ 0 & B_3^TA_2^{-1} & B_4^TA_3^{-1}
\end{bmatrix}. \label{controlLaw}
\end{equation}\label{exampleArea}
\end{example}

Notice that,  if (\ref{Garea}) is written on form (\ref{G}) the optimal controller $L_*$ has the same sparsity pattern as $B^T$.  Thus, each control input vector $u_i$ is only constructed by the states it affects in (\ref{Garea}). If we consider each subsystem $S_i$ in (\ref{Garea}) to represent an area of the physical system it models, the optimal controller (\ref{controlLaw}) is distributed according to these areas. See Fig.~\ref{fig:areaSys} for a graphical representation of the system, drawn with solid lines. Each subsystem $S_i$ is depicted by a circular node while each control input $u_i$ is given by a link connecting the subsystems it affects in (\ref{Garea}). Each disturbance $w_i$ is drawn as arrows that points toward the subsystem it affects in (\ref{Garea}). 

\begin{figure}[b] 
\centering
\includegraphics[width=0.65\linewidth]{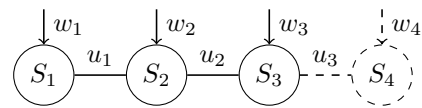}
\caption{Graphical representation of (\ref{Garea}) in solid lines. Additional subsystem $S_4$ and control input $u_3$ in dashed lines.}
\label{fig:areaSys}
\end{figure}

Consider that a fourth subsystem denoted $S_4$, of finite dimension $n_4 \geq 1$, is connected to (\ref{Garea}) via a third control input denoted $u_3$, of finite length $m_3 \geq 1$, as depicted by dashed lines in Fig. \ref{fig:areaSys}. The dynamics of subsystem $S_4$ and the altered dynamics of subsystem $S_3$ are then given by
\begin{align*}
&S_3: \quad \dot{x}_3 = A_3x_3+B_4u_2+B_5u_3+w_3\\
&S_4: \quad \dot{x}_4 = A_4x_4+B_6u_3+w_4
\end{align*}  
where matrix $A_4$ is also assumed to be symmetric and Hurwitz. Then, Theorem~\ref{theo1} is still applicable and the extended optimal controller becomes
\begin{displaymath}
L_* = \begin{bmatrix}
B_1^TA_1^{-1} & B_2^TA_2^{-1} & 0 & 0 \\ 0 & B_3^TA_2^{-1} & B_4^TA_3^{-1} & 0 \\ 0 & 0 & B_5^TA_3^{-1} & B_6^TA_4^{-1}
\end{bmatrix}.
\end{displaymath}
The expansion of the system does not alter the initial control inputs $u_1$ and $u_2$. Thus, for systems with this type of sparsity pattern, the control law $u = L_*x$ is easily scalable. Moreover, the control law is still distributed as the additional control input $u_3$ is only constructed by states $x_3$ and $x_4$.

We will now consider systems of structure (\ref{G}) with diagonal and Hurwitz matrix $A$. Then, the closed-loop system from disturbance $w$ to state $x$ with control law $u=L_*x$, from Theorem~\ref{theo1}, is internally positive by Lemma~\ref{lemIP}, given in Appendix, if and only if $-BB^T$ is Metzler. Consider the closed loop system from disturbance $w$ to output $y \coloneqq x$, with $L_* = B^TA^{-1}$, i.e.,
\begin{align*}
&\dot{x} = (A+BL_*)x+w \\ 
&y = x
\end{align*}
where $A+BL_* = A+BB^TA^{-1}$. In order to fulfil the requirements in Lemma~\ref{lemIP} we only need to check if $A+BL_*$ is Metzler as the other matrices are entry-wise non-negative. If $A$ is diagonal and Hurwitz, i.e, all diagonal elements are negative, it is necessary and sufficient that $-BB^T$ is Metzler for $A+BL_*$ to be Metzler.  

\begin{remark} 
Consider (\ref{scaled}), $A$ diagonal and $-BB^T$ Metzler. Then, diagonal $Q$ and $R$ would suffice in order for the closed-loop system from disturbance $w$ to state $x$ to be internally positive.
\end{remark}

\begin{example}
Consider three buffers of some quantity connected via links with flow $u_1$ and $u_2$ as depicted in Fig.~\ref{fig:buffersys}. The dynamics of the levels in the buffers, around some steady state depicted by the dashed lines in Fig.~\ref{fig:buffersys}, is given by 
\begin{equation}
\begin{bmatrix} \dot{x}_1\\\dot{x}_2 \\ \dot{x}_3\end{bmatrix} = \underbrace{-\textrm{diag}\left(1,2,4\right)}_A\begin{bmatrix} x_1\\x_2 \\ x_3\end{bmatrix} +\underbrace{\begin{bmatrix} -1 & 0  \\ 1 & -1  \\ 0 & 1 \end{bmatrix}}_B\begin{bmatrix} u_1\\u_2\end{bmatrix} +w. \label{waterbuffer}
\end{equation} 
State $x_i$ corresponds to the level in buffer $i=$1, 2 and 3, respectively. Each buffer has some internal dynamics dependent on its own state, as given by matrix $A$. However, with different rate coefficients for the different buffers. We want to construct a control law that minimizes the impact from disturbance $w$  on performance output $\left(x,u\right)$ in the $H_{\infty}$ norm sense. That is, we want to keep the system at its steady state, i.e, $x_i= 0$ for all $i$, while also keeping the cost down, i.e, the magnitude of the control input. 

\begin{figure}[t]
\centering
\includegraphics[width=0.6\linewidth]{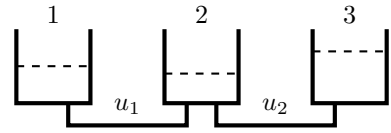}
\caption{Three buffers denoted $1$, $2$ and $3$ connected via links with flow $u_1$ and $u_2$, respectively. The dashed lines represent some steady state of the system. \label{fig:buffersys}}
\end{figure}

Given the matrix $B$ in (\ref{waterbuffer}), $-BB^T$ is Metzler. Thus the closed-loop system with the optimal control law given by Theorem~\ref{theo1}, i.e., 
\begin{displaymath}
L_* = \begin{bmatrix}
1 & -1/2 & 0 \\ 0 &1/2&-1/4  
\end{bmatrix},
\end{displaymath}
is internally positive. This implies that, in closed-loop with controller $L_*$, the states $x_i$ of (\ref{waterbuffer}) will always be non-negative, i.e., the buffer levels will never go below their steady state values, given non-negative disturbance.

To get some further intuition of what controller $L_*$ does consider control input $u_1$. It is given by $u_1= x_1-x_2/2$. Thus, $u_1$ is strictly positive if $x_1>x_2/2$ and the controller $L_*$ redistributes the quantity of buffer 1 and buffer 2 relative to their internal rate coefficients. Again, as in the previous example, $L_*$ has the same sparsity pattern as $B^T$. In this case, with one-dimensional subsystems, it means that each control input only considers local information, i.e., from the buffers it connects. \label{exampleBuff}
\end{example}

The system (\ref{waterbuffer}) can be depicted by a graph, much like the system in the previous example, Example \ref{exampleArea}. However, due to the structure of the matrix $B$ in (\ref{waterbuffer}), the links in this graph could be drawn as directed arrows. In other word, the matrix $B$ in (\ref{waterbuffer}) is the node-link incidence matrix of a directed graph. See~\cite{newman2010networks} for a formal definition of this notion. It is well known that for such matrices $B$,  the matrix product $-BB^T$ is Metzler.

\section{Coordination in the \texorpdfstring{$H_{\infty}$}{H-infinity} framework}
In this section we will extend the optimal control law given by Theorem~\ref{theo1} in order to include coordination. The problem formulation is as follows. Consider a LTI system of $\nu$ subsystems
\begin{equation}
\dot{x}_i = A_ix_i+B_iu_i+w_i, \quad i = 1,\dots, \nu  \label{Gcoord}
\end{equation} 
where $A_i$, for $i = 1,\dots, \nu$, is symmetric and Hurwitz. Furthermore, the control inputs $u_i$ have to coordinate in order to fulfil the following constraint
\begin{equation}
u_1+u_2+\dots+u_{\nu}=0. \label{constraint}
\end{equation}
Given the performance objective $(x,u)$ and the coordination constraint in (\ref{constraint}) we want to construct an optimal \texorpdfstring{$H_{\infty}$}{H-infinity} static state feedback controller for (\ref{Gcoord}). 

Our solution to the given problem is as follows. Rewrite control input $u_1$ in terms of the other control inputs given (\ref{constraint}), i.e., 
\begin{equation}
u_1=-u_2-u_3\,\dots \,-u_{\nu}, \label{u1}
\end{equation}
and define $\tilde{u} = [u_2, \,u_3,\,\dots ,\,u_{\nu}]^T$. Then, 
\begin{displaymath}
u = \underbrace{\begin{bmatrix}
-\mathbf{1}^T_{\nu-1} \\ I_{\nu-1}
\end{bmatrix}}_{D} \tilde{u}
\end{displaymath}
and the overall system of (\ref{Gcoord}) can be written
\begin{displaymath}
\dot{x} = \underbrace{\textrm{diag}(A_1,\dots,A_{\nu})}_Ax+\underbrace{\textrm{diag}(B_1,\dots,B_{\nu})}_BD\tilde{u} +w 
\end{displaymath}
with performance output $(x,u)=(x,D\tilde{u})$. Now, define $$R = D^TD = I+\mathbf{1} \mathbf{1}^T,$$ as described in Remark 1, and notice that $R^{-1}=  I-\frac{1}{\nu}\mathbf{1}\mathbf{1}^T$. The optimal control law by Theorem~\ref{theo1} is then 
\begin{align*}
\tilde{u} &= R^{-1}D^TB^TA^{-1}x  \\&= \left(I_{\nu-1}-\frac{1}{\nu}\mathbf{1}_{\nu-1}\mathbf{1}_{\nu-1}^T \right)\begin{bmatrix}
-\mathbf{1}^T_{\nu-1} \\ I_{\nu-1}
\end{bmatrix}^TB^TA^{-1}x \\ &= \left(\begin{bmatrix}
0& I_{\nu-1}
\end{bmatrix}-\frac{1}{\nu}\mathbf{1}_{\nu-1}\mathbf{1}_{\nu}^T \right)B^TA^{-1}x.
\end{align*}  
Thus, $u_i$ for $i=2,\,\dots,\,\nu$, i.e., the elements in $\tilde{u}$, is 
\begin{equation}
u_i = B_i^TA_i^{-1}x_i-\frac{1}{\nu}\sum_{k=1}^{\nu}B_k^TA_k^{-1}x_k. \label{controllawutilde}
\end{equation}
Now, consider $u_1$ again, 
\begin{align*}
u_1 \stackrel{\textrm{(\ref{u1})}}{=}-\sum_{i=2}^{\nu}u_i \stackrel{\textrm{(\ref{controllawutilde})}}{=} -\sum_{i=2}^{\nu}\left( B_i^TA_i^{-1}x_i-\frac{1}{\nu}\sum_{k=1}^{\nu}B_k^TA_k^{-1}x_k \right)\\
= -\left(\sum_{k=1}^{\nu}B_k^TA_k^{-1}x_k-B_1^TA_1^{-1}x_1-\frac{v-1}{v}\sum_{k=1}^{\nu}B_k^TA_k^{-1}x_k \right) \\ 
= B_1^TA_1^{-1}x_1-\frac{1}{\nu}\sum_{k=1}^{\nu}B_k^TA_k^{-1}x_k,
\end{align*}
i.e., it has the same structure as (\ref{controllawutilde}). Thus, the optimal control law can be written
\begin{equation}
u_i =  B_i^TA_i^{-1}x_i-\frac{1}{\nu}\sum_{k=1}^{\nu}B_k^TA_k^{-1}x_k\label{coordlaw}
\end{equation}
for each subsystem $i=1,\,\dots,\,\nu$ in (\ref{Gcoord}). The first term of $u_i$ in (\ref{coordlaw}) is a local term, only dependent upon the subsystem $i$,  while the second term is dependent on global information of the overall system. However, as this term is equal for all control inputs $u_i$, (\ref{coordlaw}) might still be appropriate for distributed control use. 

In \cite{madjidian2014distributed}, a similar type of problem is considered, however in the $H_2$ framework with stochastic disturbances and the necessity of homogeneous subsystems. The optimal control law derived in \cite{madjidian2014distributed} and the one we suggest in (\ref{coordlaw}) are similar in structure. However, our approach can treat heterogeneous systems in addition to homogeneous ones. On the contrary, it is only applicable to systems with symmetric and Hurwitz state matrix, properties that are not necessary in \cite{madjidian2014distributed}.

\section{Numerical example}
Consider a system of the same structure as (\ref{bufferSys}) given in Section I, i.e., a system
\begin{align}  \label{numEx}
\begin{bmatrix} \dot{x}_1\\\dot{x}_2 \\ \dot{x}_3\end{bmatrix} = &\underbrace{-\begin{bmatrix} a_1 & 0 & 0 \\ 0 & a_2 & 0 \\ 0 & 0 & a_3\end{bmatrix}}_A\begin{bmatrix} x_1\\x_2 \\ x_3\end{bmatrix} \nonumber \\&+\underbrace{\begin{bmatrix} -b_1 & 0 & 0  \\ b_2 &b_3 & -b_4  \\ 0&0 & b_5 \end{bmatrix}}_B\begin{bmatrix} u_{12}\\u_2\\u_{23}\end{bmatrix} + \begin{bmatrix} w_1\\w_2 \\ w_3\end{bmatrix}
\end{align} 
where $a_i>0$, for $i=$1, 2 and 3, and $b_j>0$, for $j = 1, \dots, \,5$, and the performance output is $\left(x,\,u\right)$. We will now compare the optimal controller given by Theorem~\ref{theo1}, i.e., $L_*$, and an optimal controller derived by the ARE-approach, see~\cite{zhou1996robust}, denoted $L_{G}$ for global. In the latter approach, we consider the minimal value of the \texorpdfstring{$H_{\infty}$}{H-infinity} norm of (\ref{Gtransfer}) given by Theorem~\ref{theo1} and iterate over the ARE-constraint until this minimal value is reached. See \cite{MATLAB} for the software used. Controllers $L_1$ and $L_2$ given in Section I are examples of controllers $L_*$ and $L_G$ treated here, respectively. 

Controllers $L_*$ and $L_G$ are optimal and thus they both obtain the minimal value of the $H_{\infty}$ norm of (\ref{Gtransfer}). Now we want to compare how they affect the closed-loop dynamics more in detail. We randomly generate values of the parameters $a_i$ and $b_j$ in $(0.1,5]$ and compare the step-response of the states of (\ref{numEx}) in closed-loop with $L_*$ and $L_G$. In other words, given constant disturbance of value 1. The average dynamics over 50 such randomly generated systems is shown in Fig.~\ref{fig:numEx}. To clarify, we average over the absolute value of the step response in each time instance.  

\begin{figure}[b] 
\centering
\includegraphics[width=0.9\linewidth]{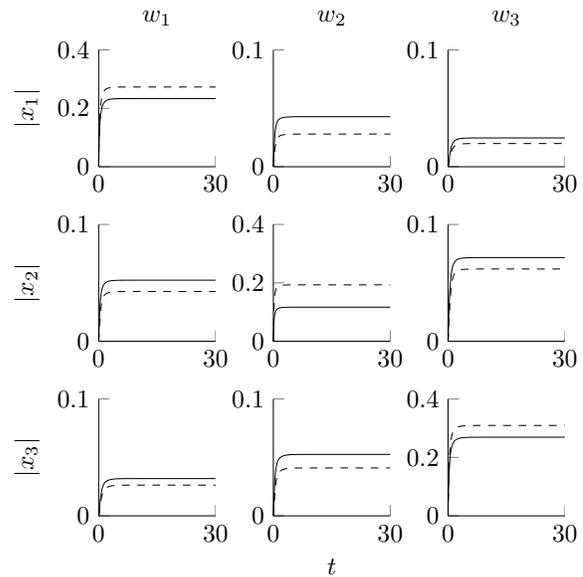}
\caption{Average step response for states $x_1$, $x_2$ and $x_3$ for closed-loop systems with controller $L_*$ (solid lines) and $L_G$ (dashed lines).} \label{fig:numEx}
\end{figure}

\begin{figure}[t] 
\centering
\includegraphics[width=0.6\linewidth]{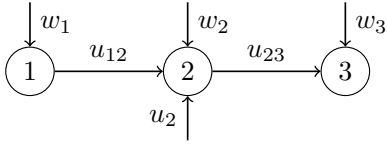}
\caption{Associated graph of~(\ref{numEx}).}
\label{fig:buffersystem}
\end{figure}

The system (\ref{numEx}) can be depicted by the graph given in Fig.~\ref{fig:buffersystem}, as described in Section III. If we compare the step responses shown in Figure~\ref{fig:numEx}, it seems as if controller $L_*$ is better at attenuating local disturbances than $L_G$ is. With local disturbances we mean the disturbance that points towards the state in  Fig.~(\ref{fig:buffersystem}). However, this is at the expense of larger impact on distance. For instance, consider  disturbance $w_1$. Its impacts on state $x_1$ is lower for controller $L_*$ than $L_G$ while its impact on the remaining states is the opposite for the given controllers. However, overall they are comparable in performance.

We will end this numerical example by commenting on controller $L_2$ given in Section I, that is an example of controller $L_G$ treated in this numerical example. Some entries of $L_2$ are small in magnitude compared to the other entries, i.e., entries (2,1), (2,3) and (3,1), where the first number in each parenthesis is the row and the second is the column. However, only entry (3,1) can be replaced with a zero for the controller to still achieve the optimal bound. Furthermore, for systems of much larger dimension than (\ref{bufferSys}), this type of reduction analysis might be difficult. 

\section{CONCLUSIONS}

We give a simple form for an optimal  $H_{\infty}$ static state feedback law applicable to LTI systems with symmetric and Hurwitz state matrix. This simple form is given in the matrices of the system's state space representation which makes it transparent. The structure of the control law also simplifies synthesis and enables scalability, especially given sparse systems. Furthermore, given compatible system sparsity patterns the control law is distributed. The examples we give consider diagonal or block diagonal state matrices and somewhat more general sparsity patterns of the remaining system matrices. These types of system sparsity patterns are common among the systems for which distributed control methods are needed. Furthermore, we extend the optimal control law in order to incorporate coordination among subsystems. The resulting coordinated control law is similar for all subsystems. More specifically, for each subsystem, it is a superposition of a local term and an averaged centralized term where the latter is equal for all subsystems involved in the coordination. In conclusion, our control law is well suited for distributed control purposes. Future research directions include to consider saturation constraints on the optimal control law as such are common in the systems intended for its application. Furthermore, to investigate the existence of an analogous optimal control law given output feedback instead of state feedback.  


\addtolength{\textheight}{-12cm}   



\section*{APPENDIX}

\begin{lem} \textit{The Kalman-Yakubovich-Popov lemma} \\
Given $A \in \mathbb{R}^{n \times n}$, $B \in \mathbb{R}^{n \times m}$, $M=M^T \in \mathbb{R}^{(n+m)\times(n+m)}$, with $\textrm{det}(j\omega I-A) \neq 0$ and $(A,B)$ controllable, the following two statements are equivalent:\\
(i)
\begin{displaymath}
\begin{bmatrix}
\left(j \omega I-A \right)^{-1}B \\ I 
\end{bmatrix}^*M \begin{bmatrix}
\left( j \omega I-A \right)^{-1} B \\ I 
\end{bmatrix} \preceq 0
\end{displaymath}
$\forall \omega \in \mathbb{R}\cup\{\infty \} $.\\
(ii) There exists a matrix $P \in \mathbb{R}^{n \times n}$ such that $P=P^T$ and 
\begin{displaymath}
M+ \begin{bmatrix}
A^TP + PA & PB \\ B^TP & 0
\end{bmatrix} \preceq 0 
\end{displaymath}
The corresponding equivalence for strict inequalities holds even if $(A,B)$ is not controllable. \label{KYP}
\end{lem}
\begin{proof}
See ~\cite{rantzer1996kalman}.
\end{proof}
\textit{Remark.} If the upper left corner of $M$ is positive semidefinite, it follows from (\ref{KYP}) and Hurwitz stability of $A$ that $P \succeq 0$ ~\cite{rantzer1996kalman}.

\begin{lem} 
The LTI system 
\begin{align*}
\dot{x} = Ax + Bv\\
y = Cx + Dv
\end{align*}
is internally positive if and only if 
\begin{enumerate}
\item[i] $A$ is Metzler, and 
\item[ii] $B \geq 0$, $C \geq 0$ and $D \geq 0$.
\end{enumerate} \label{lemIP}
\end{lem}
\begin{proof}
See \cite{kaczorek2001externally}.
\end{proof}


\bibliographystyle{IEEEtran}
\bibliography{IEEEabrv,ref}

\end{document}